\documentclass{amsart}
\usepackage[utf8]{inputenc}
\usepackage{a4wide}
\usepackage{amsmath,amsxtra,amssymb,latexsym, amscd,amsthm,nicefrac,mathrsfs,mathtools}
\usepackage{indentfirst}
\usepackage[symbol]{footmisc} 
\usepackage{dutchcal}

\usepackage{tikz}

\DeclarePairedDelimiter\floor{\lfloor}{\rfloor}
 \newtheorem{theorem}{Theorem}[section]
 
 \newtheorem{lemma}[theorem]{Lemma}
 \newtheorem{proposition}[theorem]{Proposition}
 \theoremstyle{definition}
 
 \theoremstyle{remark}

 \numberwithin{equation}{section}

\begin{document}
\newcommand{\onehalf}{\tfrac{1}{2}}
\newcommand{\norm}[1] {\| #1 \|}
\newcommand{\lrnorm}[1]{\left\| #1 \right\|}
\newcommand{\bignorm}[1]{\bigl\| #1 \bigr\|}
\newcommand{\Bignorm}[1]{\Bigl\| #1 \Bigr\|}
\newcommand{\Biggnorm}[1]{\Biggl\| #1 \Biggr\|}
\newcommand{\biggnorm}[1]{\biggl\| #1 \biggr\|}
\newcommand{\mathfrakA}{\mathpalette\bigmathfrakA\relax}
\newcommand{\bigmathfrakA}[2]{\scalebox{1.4}{$#1\mathfrak{a}$}}
\newcommand{\calA}  {\mathfrakA}
\newcommand{\NN} {\mathbb N}
\newcommand{\ZZ} {\mathbb Z}
\newcommand{\CC} {\mathbb C}
\newcommand{\RR} {\mathbb R}
\newcommand{\QQ} {\mathbb Q}
\newcommand{\dx} {{\,\mathrm{d}x}}
\newcommand{\dy} {{\,\mathrm{d}y}}
\newcommand{\du} {{\,\mathrm{d}u}}
\newcommand{\dr} {{\,\mathrm{d}r}}
\newcommand{\ds} {{\,\mathrm{d}s}}
\newcommand{\dt} {{\,\mathrm{d}t}}
\newcommand{\dxi} {{\,\mathrm{d}\xi}}
\newcommand{\calD}{{\mathscr D}}
\newcommand{\calH}{{\mathcal H}}
\newcommand{\DOMAIN}{\mathscr D} 
\newcommand{\BOUNDED}{\mathscr L}
\newcommand{\dual}[2]{ \left\langle #1, #2 \right\rangle}
\newcommand{\sprod}[2]{ \left[  #1 , #2 \right]}
\newcommand{\VECT}[2]{ \begin{pmatrix}#1\\#2\end{pmatrix}}
\newcommand{\LANDAUo}{\mathcal o}
\newcommand{\LANDAUO}{\mathcal O}
\newcommand{\cals}{\mathcal s}
\newcommand{\Ce}{{\mathrm C}}
\newcommand{\SUCHTHAT}{:\;}
\newcommand{\EVOLUTION}{U(t, s)_{0\le s\le t}}
\newcommand{\DUALEVOLUTION}{U_{-1}^*(t, s)_{0\le s\le t}}

\setlength{\parindent}{0cm}

\title{Exact observability of a 1D wave on a non-cylindrical domain } 

\author{Bernhard Haak} 
\address{IMB // Univ. Bordeaux // 351 cours de la Liberation // 33405
  Talence} 
\thanks{The first named author is partially supported by
  ANR project  ANR-12-BS01-0013 'Harmonic Analysis at its Boundaries'.}  

\author{Duc-Trung Hoang} 

\thanks{The second named author kindly acknowledges the financial
  support of his PhD thesis at Bordeaux University.}

\date{\today} 
\allowdisplaybreaks

\begin{abstract}
  We discuss admissibility and exact observability estimates of
  boundary observation and interior point observation of a
  one-dimensional wave equation on a time dependent domain for
  sufficiently regular boundary functions. We also discuss moving
  observers inside the non-cylindrical domain and simultaneous
  observability results.
\end{abstract}

\maketitle

\section{Introduction and main results}

\bigskip

\begin{minipage}{0.65\linewidth}
In this article we are concerned with exact observability
of the 1D wave equation on a domain with time-dependent boundary. 
To be precise, let  $s : \RR_+ \to (0, \infty)$  and let
\[
\Omega = \bigl\{ (x, t)\in \RR^2 :\quad   t \ge 0 \;\text{and}\; 0 \le x \le s(t) \bigr\},
\]
Where $s(0)=1$ and $\norm{ s'(t) }_{L_\infty(\RR)} < 1$. The last
condition ensures amongst other things that the characteristic
emerging from the origin hits the boundary in finite time.
Let $f \in L_2([0,1])$ and $g \in H_0^1([0,1])$ be initial values.  We
consider a wave equation on $\Omega$ with Dirichlet boundary
conditions
\begin{equation}
  \label{eq:1D-wave-eq-linear-bdry}\tag{W.Eq}
  \left\{
    \begin{array}{ll}
    u_{tt} - u_{xx} = 0  \qquad                        & (x, t) \in \Omega\\
    u(0, t) = u(s(t), t) = 0 \qquad     & t \ge 0\\
    u(x,0) = g(x)              \qquad                & x \in [0,1]\\
    u_t(x,0) = f(x)            \qquad                & x \in [0,1]
    \end{array}
  \right.
\end{equation}
\end{minipage}\hspace*{0.5cm}
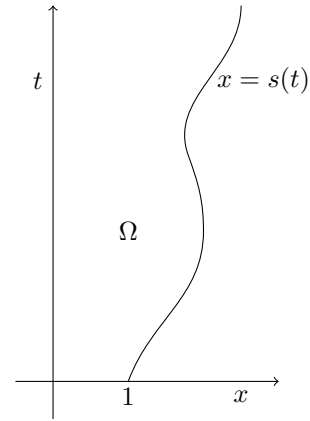
\begin{minipage}{0.3\linewidth}
  \begin{tikzpicture}
  \draw[->] (-.5,0) -- (3,0) ;    \path (2.5, -.2) node {$x$} ; 
  \draw[->] (0,-.5) -- (0,5) ;    \path (-.2, 4) node {$t$} ; 
  \path (1, -.2) node {$1$} ;
  \draw (1, 0) to[out=70,in=270] (2,2)  to[out=90,in=290] (1.8, 3) to[out=110,in=270] (2.5,5) ; 
  \path (1, 2) node {$\Omega$} ;
  \path (2.8, 4) node {$x=s(t)$} ;
\end{tikzpicture}
\end{minipage}

\subsection{Existence of solutions}

There are several natural approaches to
(\ref{eq:1D-wave-eq-linear-bdry}).  One may for example transform the
domain $\Omega$ to a cylindrical domain. Instead, seeking a natural
and more simple approach, we try to develop the solution $u$ into a
series of the form
\begin{equation} \label{eq:solution-as-series} 
  u(x,t) := \sum_{n \in \ZZ}  A_n  \Bigl(e^{2 \pi i n \; \varphi(t{+}x))}  - e^{2\pi i n \; \varphi(t{-}x))}\Bigr)
\end{equation}
where the coefficients $A_n$ are given by the initial data $(g, f)$.
This approach has almost a century of history, dating back to Nicolai
\cite{Nicolai} in the case of a linear moving boundary
$s(t)=1+\varepsilon t$ and Moore \cite{Moore} for general boundary
curves (however only asymptotic developments for $\varphi$ are given).
We refer to Donodov \cite{Dodonov} for a large number of references.
In order to satisfy the Dirichlet boundary condition, we need a
solution $\varphi$ to the functional equation
\begin{equation}
  \label{eq:functional-equation}
    \varphi( t + s(t) )   - \varphi( t - s(t) ) = 1.
\end{equation}
Because of the importance of this functional equation we fix the
notation $\alpha(t) := t + s(t)$ and $\beta(t) := t - s(t)$ and mention
that both are strictly increasing bijections from $\RR_+$ to
$[\pm s(0), \infty)$, respectively.  We will also consider
$\gamma = \alpha \circ \beta^{-1}: [-s(0), \infty) \to [s(0),
\infty)$.
Most solutions to (\ref{eq:functional-equation}) are useless for our
purposes\footnote{It is indeed easy to construct solutions depending
  on an arbitrary function by using the axiom of choice}. On the other
hand side, under reasonable assumptions on the boundary function,
differentiable solutions to (\ref{eq:functional-equation}) are unique,
at least up to an additive constant. This is of course what we look
for.  In some easy cases a differentiable solution $\varphi$ can be
found by calculus, see the following table for some examples. We refer
to a detailed discussion on the general situation in the
appendix~\ref{sec:appendix}.

\medskip
\begin{tabular}{llll}
  Name  & Boundary function &  &  Solution to (\ref{eq:functional-equation}) \\
\hline\\ 
  linear moving boundary 
 & $s(t) = 1+\varepsilon t$ & $\varepsilon \in (0, 1)$ 
 & $\varphi(t) =  \ln(\tfrac{1+\varepsilon}{1-\varepsilon})^{-1} \, \ln( 1 {+} \varepsilon t )$ \\ 
  parabolic boundary  
 & $s(t) = \sqrt{1+\varepsilon t}$ & $\varepsilon \in (0,2)$ 
 & $\varphi(t) = \tfrac1{2\varepsilon} \sqrt{\varepsilon^2 + 4 \varepsilon t + 4}$\\
  hyperbolic boundary
 &  $s(t) = \tfrac{1}{\varepsilon}(-1 + \sqrt{1{+}(1{+}\varepsilon t)^2})$ &  $\varepsilon >0$ 
 & $\varphi(t) =  \tfrac{\varepsilon t}{1+\varepsilon t}$ \\
  shrinking domain 
 & $s(t) = \frac1{1+\varepsilon t}$ &  $\varepsilon \in (0, 1)$ 
 &  $\varphi(t) =\tfrac{\varepsilon}4 (t+\tfrac1\varepsilon)^2$ .
\end{tabular}
\medskip

For simplicity of notation, we shall always assume $s(0)=1$ ; in case
of hyperbolic boundaries some straight-forward modifications have to
be made.  The common denominator of these examples is the following:
$\varphi \in \Ce^2([-1, \infty))$ and $\varphi'(t) > 0$ for all
$t \ge -1$. We call $s$ an {\bf admissible boundary function} if
(\ref{eq:functional-equation}) admits such a solution $\varphi$.
  
\begin{proposition}\label{prop:proof-of-series-formula}
  Let $s$ be an admissible boundary function and assume the initial
  data $f, g \in \calD((0,1))$. Then $(g, f)$ determine uniquely a
  sequence $(A_n)_{n\in \ZZ} \in \ell_2$ such that for $t \ge 0$ and
  $0 \le x \le s(t)$, the function (\ref{eq:solution-as-series}) is
  the solution of the moving boundary wave equation {\rm
    (\ref{eq:1D-wave-eq-linear-bdry})}.
\end{proposition}

We start the proof with the following trivial observation.
\begin{lemma}\label{lem:ONS}
  For fixed $t_0 \geq 0$, the family $\{e^{2\pi i n \, \varphi(x)} : n \in \ZZ\}$,
  is a complete orthonormal system in $H := L_2([t_0{-}s(t_0), t_0{+}s(t_0)], \varphi'(x)\dx)$.
\end{lemma}

For $t_0{=}0$, we obtain as a particular case that the family $(b_n)$
with $b_n(x) = e^{2\pi i n \, \varphi(x)}$ is an orthonormal basis in
$H := L_2([-1,1], \varphi'(x)\dx)$. Since there is $C>0$ such that
$\tfrac1C \le \varphi'(x) \le C$ on $[0,1]$, we have
$L_2([-1,1], \varphi'(x)\dx) = L_2([-1,1], \dx)$ as sets with
equivalent respective norms\footnote{In particular, $(b_n)$ is a Riesz basis in
$ L_2([-1,1])$.}. 

\begin{proof} [Proof of Proposition~\ref{prop:proof-of-series-formula} ]
We let $F(x) = \int_{0}^{x}f(s)\ds$ and
\[
h(x) := \tfrac{1}{2} \cdot \left\{
  \begin{array}{rcllr}
    g(x)  & +  &F(x)  \hspace*{1cm}  & \text{ for } & 0  \le  x \le 1 \\
   -g(-x) & +  &F(-x)                & \text{ for } &  -1  \le  x < 0
  \end{array}
\right.
\]
By assumption, $h\in H$ that we develop into the orthonormal basis:
$h = \sum_{\ZZ} \; \langle h, b_n \rangle\, b_n $. We shall always note
\begin{equation}  \label{eq:def-of-An}
 A_n = \langle h, b_n \rangle = \int_{-1}^{1} h(x) e^{2\pi i n \, \varphi(x)}\varphi'(x) \dx  
\end{equation}
Since $g(0){=}g(1){=}0$, we have $h(1){=}h(-1)$ so that
$h \in H_0^1([-1,1])$. Hence the sequences $(A_n)$ and $(n\,A_n)$ are
square-summable.  Taking sum and difference, we may develop $g$ and
$F$ as follows:
\[ 
F(x)  =  \sum_{n \in \ZZ} A_n\Bigl(e^{\frac{i\pi n}{\eta_\varepsilon}\varphi(x)} + e^{\frac{i\pi n}{\eta_\varepsilon}\varphi(-x)}\Bigr),
\qquad\qquad x\in [0,1]
\]
and
\[
g(x) = \sum_{n \in \ZZ} A_n \Bigl(e^{\frac{i\pi n}{\eta_\varepsilon}\varphi(x)} - e^{\frac{i\pi n}{\eta_\varepsilon}\varphi(-x)}\Bigr),
\qquad\qquad x\in [0,1].
\] 
  Since we suppose  $f , g \in \calD((0,1))$, $h$
  satisfies the periodicity condition
  $h^{(\alpha)}(-1){=}h^{(\alpha)}(1)$ for all derivative orders
  $\alpha\ge 0$.  As a consequence, the series of $F$, $g$ and $h$
  above may be differentiated term by term. We let
\[
  u(x,t) := \sum_{n \in \ZZ}  A_n  \Bigl(e^{2 \pi i n \; \varphi(t{+}x))}  - e^{2\pi i n \; \varphi(t{-}x))}\Bigr)
\]
Since $\varphi \in \Ce^2([-1, \infty))$, $u$ is twice differentiable
and with respect to $x$ and $t$. Moreover, partial derivatives can be
calculated term by term. As an immediate consequence,
$u_{xx} - u_{tt} = 0$ in the interior domain $\Omega^\circ$. Moreover, $u$ satisfies the
Dirichlet condition since for $x=0$
\[
   u(0,t) = \sum_{n \in \ZZ} A_n \Bigl(e^{2 \pi i n \; \varphi(t))}  - e^{2\pi i n \; \varphi(t))} \Bigr) = 0
\]
whereas for $x=s(t)$, thanks to the functional equation (\ref{eq:functional-equation}), 
\begin{align*}
u( s(t) , t) 
= & \; \sum_{n \in \ZZ} A_n \Bigl( e^{2 \pi i n \; \varphi(t+s(t))}  - e^{2\pi i n \; \varphi(t-s(t))} \Bigr) \\
= & \; \sum_{n \in \ZZ} A_n  e^{2 \pi i n \; \varphi(t+s(t))}  \bigl(1 - e^{2\pi i n}  \bigr) = 0. \qedhere
\end{align*}
\end{proof}

The series representation of the solution is the key to obtain
explicit and precise constants for admissibility and exact
observability in different situations, since they can be played back
to classical Fourier analysis.

Let us fix some often appearing constants:
\begin{equation}  \label{eq:def-of-constants-c-and-C}
  \begin{split}
   m(t) = & \; \min\{ \varphi'(x)\SUCHTHAT  x \in [t-s(t), t+s(t)] \} \qquad \text{and}\qquad \\
   M(t) = & \; \max\{ \varphi'(x)\SUCHTHAT  x \in [t-s(t), t+s(t)] \}.
  \end{split}
\end{equation}
Since on $[0,1]$, $m(0) \le \varphi'(x) \le M(0)$, we may use the
unweighted Poincaré inequality on $[0,1]$ to show that
  \begin{equation}
    \label{eq:norm-def}
\bignorm{ (g,f)}_{H_0^1([0,1];\frac{\dx}{\varphi'(x)}) \times L_2 ([0,1];\frac{\dx}{\varphi'(x)})}^2 
:= \bignorm{\nabla g}_{L_2([0,1];\frac{\dx}{\varphi'(x)})}^2 +  \bignorm{f}_{L_2([0,1];\frac{\dx}{\varphi'(x)})}^2.
\end{equation}
is an equivalent to
$\norm{g}_{L_2 ([0,1];\frac{\dx}{\varphi'(x)})}^2 + \norm{g'}_{L_2
  ([0,1];\frac{\dx}{\varphi'(x)})}^2 + \norm{f}_{L_2
  ([0,1];\frac{\dx}{\varphi'(x)})}^2$.
The notation
\[
\norm{ (g, f)}_{H_0^1 \times L_2}^2 := \norm{ g' }_{L_2(0,1)}^2 +
\norm{ f }_{L_2(0,1)}^2
\]
(without specifying intervals or weights) always refers to the
unweighted norms on $[0,s(0)] = [0,1]$.

\begin{proposition}\label{prop:norm-equivalence-f-g}
  We have the
  following estimate
\[ 
                 8\pi^2 m(0) \; \sum_{n\in \ZZ}n^2|A_n|^2 
\quad \leq \quad \norm{(g,f)}_{H_0^1 \times L_2 }^2 
\quad \leq \quad 8\pi^2 M(0)  \; \sum_{n\in \ZZ}n^2|A_n|^2,
\]
where the constants are given by (\ref{eq:def-of-constants-c-and-C}).
\end{proposition}
\begin{proof}
Recall that $g(x) = h(x) - h(-x)$ and $F(x) = h(x)+h(-x)$ on $[0,1]$. Therefore
\begin{align*}
   \; \bignorm{ (g,f)}_{ H_0^1 \times L_2}^2 
= & \; \bignorm{ g' }_{L_2([0,1])}^2 + \bignorm{ F' }_{L_2([0,1])}^2\\
= & \; \bignorm{ h'(\cdot) + h'(-(\cdot)) }_{L_2([0,1])}^2 +  \bignorm{ h'(\cdot) - h'(-(\cdot)) }_{L_2([0,1])}^2\\
= & \; 2 \bignorm{  h' }^2_{L_2([0,1])} + 2 \bignorm{  h'(-\cdot) }^2_{L_2([0,1])} = 2 \bignorm{ h' }_{L_2([-1,1])}^2
\end{align*}
by parallelogram identity. Estimating the maximum of $\varphi'$ and $\frac{1}{\varphi'}$ on $[-1,1]$ 
allows to relate $\bignorm{ h' }_{L_2([-1,1], \varphi'(x)\dx)}^2$ and $\bignorm{ h' }_{L_2([-1,1])}^2$, and the result follows
by Parseval's identity.
\end{proof}

Observe that for the concrete examples we discuss later, the minimum
respectively maximum is easy to calculate; we obtain therefore
explicit constants in Proposition~\ref{prop:norm-equivalence-f-g}.

\subsection{Energy estimates}\label{subsec:energy}
Define the energy of the problem (\ref{eq:1D-wave-eq-linear-bdry}) as
\[
   E_u(t) = \onehalf\int_{0}^{s(t)} |u_x(x, t)|^2+|u_t(x, t)|^2 \dx.
\]
for all $t \geq 0$. When $t = 0$, we see that
$E_u(0) = \onehalf\norm{(g,f)}_{H_0^1 \times L_2 (0,1)}^2$. In the
case of a 1D-wave equation with time-invariant boundary
(i.e. $s \equiv 1$) the energy is constant. In time-dependent
domains it decays when $s'(t) > 0$ and increases when $s'(t) < 0$.

\begin{lemma}\label{lem:energy--derivative}
The function $t \mapsto E_u(t)$ is decreasing for $t \geq 0$ if $s'(t) > 0$ and increasing when $s'(t) < 0$. More precisely,
\begin{equation} \label{eq:energy-derivative} 
\tfrac{d}{dt}E_u(t) =   \, \tfrac{s'(t)}{2}(s'(t)^2 - 1)\; |u_x(s(t), t)|^2.
\end{equation}
\end{lemma}
\begin{proof}
  Differentiating the constant zero function $u(s(t),t)$ with respect
  to $t$ yields
  $u_t(s(t),t) = - s'(t)\; u_x(s(t),t)$.   We use this twice  in the following calculation.
\begin{align*}
\tfrac{d}{dt}E_u(t) 
=  & \; \tfrac{1}{2}s'(t)(u_t^2+u_x^2)\bigl|_{x= s(t)} \;+\; \tfrac{1}{2} \int_{0}^{s(t)}\tfrac{\partial}{\partial t}(u_t^2+u_x^2) \dx   \\
=  & \; \tfrac{s'(t)}{2} (1{+}s'(t)^2) \; (u_x^2)\bigl|_{x= s(t)} \;+\; \int_{0}^{s(t)}(u_tu_{tt}+u_xu_{tx}) \dx \\
=  & \; \tfrac{s'(t)}{2} (1{+}s'(t)^2) \; (u_x^2)\bigl|_{x= s(t)} \;+\; \int_{0}^{s(t)}(u_tu_{xx}+u_xu_{tx}) \dx\\{}
\text{(integration by parts)}\quad
=  & \; \tfrac{s'(t)}{2} (1{+}s'(t)^2) \; (u_x^2)\bigl|_{x= s(t)} \;+\; \Bigl[ u_tu_x \Bigr]_{x= 0}^{x= s(t)}  \\
=  & \; \tfrac{s'(t)}{2} (1{+}s'(t)^2) \; (u_x^2)\bigl|_{x= s(t)} + u_tu_x\bigl|_{x= s(t)} \\
= & \;  \tfrac{s'(t)}{2} (s'(t)^2-1) \;  |u_x(s(t), t)|^2.
\end{align*}
Recall that $\norm{ s' }_\infty < 1$ to conclude that $\text{sign}(\tfrac{d}{dt}E_u(t)) = - \text{sign}(s'(t))$.
\end{proof}

\begin{proposition}\label{prop:enery-estimate}
 For (\ref{eq:1D-wave-eq-linear-bdry}) the following energy estimate holds
\begin{equation}
\tfrac{ m(t)}{2M(0)}
\; \bignorm{ (g, f) }_{H_0^1 \times L_2}^2
\quad \leq \quad E_u(t)  \quad \leq \quad 
\tfrac{M(t)}{2m(0)}
\; \bignorm{ (g, f) }_{H_0^1 \times L_2}^2
\end{equation}
where the constants are given by (\ref{eq:def-of-constants-c-and-C}).
\end{proposition}

\begin{proof}
Taking term by term derivatives in (\ref{eq:solution-as-series}) gives
\begin{align*}
u_x(x,t)  = & \; 2\pi i \sum_{n \in \ZZ} n A_n  \bigl(\varphi'(t{+}x)e^{2\pi i n \, \varphi(t+x)} + \varphi'(t{-}x)e^{2\pi i n \, \varphi(t-x)}\bigr)\\
u_t(x,t)  = & \; 2\pi i \sum_{n \in \ZZ} n A_n  \bigl(\varphi'(t{+}x)e^{2\pi i n \, \varphi(t+x)} - \varphi'(t{-}x)e^{2\pi i n \, \varphi(t-x)}\bigr)
\end{align*}
Therefore, using parallelogram identity as in the proof of
Proposition~\ref{prop:norm-equivalence-f-g},
\begin{align*}
2 E_u(t) =  & \; \int_{0}^{s(t)}\bigl|u_x(x,t)\bigr|^2 + \bigl|u_t(x,t)\bigr|^2 \dx\\
= & \; 8\pi^2  \Bigl(\int_{0}^{s(t)}\Bigl| \sum_{n\in \ZZ} n A_n  \varphi'(t{+}x) e^{2\pi i n \, \varphi(t+x)} \Bigr|^2 \dx 
   \; +  \;  \int_{0}^{s(t)} \Bigl| \sum_{n\in \ZZ} n A_n \varphi'(t{-}x) e^{2\pi i n \, \varphi(t-x)} \Bigr|^2 \dx \Bigr)\\ 
= & \; 8 \pi^2  \int_{t-s(t)}^{t+s(t)}\Bigl| \sum_{n\in \ZZ} n A_n \bigl(\varphi'(y) e^{2\pi i n \, \varphi(y)}\bigr)\Bigr|^2 \dy.
\end{align*}
This yields the double inequality
\[
 4\pi^2 m(t) \;  a(t) 
\leq   \; E_u(t) \;  
\leq   \;  4\pi^2 M(t) \; a(t)
\]
where
\[ 
    a(t) = \int_{t-s(t)}^{t+s(t)}\Bigl| \sum_{n\in \ZZ}n A_n e^{2\pi i n \, \varphi(y)}\Bigr|^2\varphi'(y) \dy.
\]
By Lemma~\ref{lem:ONS} and Proposition~\ref{prop:norm-equivalence-f-g} we conclude.
\end{proof}

\section{Point Observations}

\subsection{Boundary Observation}
Recall the notation $\alpha(t)=t+s(t)$, $\beta(t) = t -s(t)$ and
$\gamma = \alpha \circ \beta^{-1}$.

\begin{theorem}\label{thm:obs-on-left-bdry}
  For any admissible boundary curve $s(t)$ and solution $u$ to the
  moving boundary wave equation {\rm     (\ref{eq:1D-wave-eq-linear-bdry})} given by
  (\ref{eq:solution-as-series}) the following double inequality holds:
\begin{equation}\label{eq:thm-obs-on-left-bdry}
2 \tfrac{m(\beta^{-1}(0))}{M(0)}  \, \bignorm{(g,f)}_{H_0^1 \times L_2}^2
\quad \leq \quad
   \int_0^{\gamma(0)} \bigl|u_x(0,t) \bigr|^2 \, \dt
\quad \leq \quad
2 \tfrac{M(\beta^{-1}(0))}{m(0)} \, \bignorm{(g,f)}_{H_0^1 \times L_2}^2
\end{equation}
In particular, with the observations $C \psi = \psi_x(0)$ the problem {\rm
  (\ref{eq:1D-wave-eq-linear-bdry})} is exactly observable in time
$\tau$ if and only if $\tau \geq \gamma(0)$.
\end{theorem}
\begin{proof}
 Differentiating   $u$ term by term, and evaluating at $x=0$ we have for all $\tau >0$
\[
\norm{u_x(0,t)}_{L_2(0,\tau,\frac{1}{\varphi'(t)})}  
= \; \int_{0}^{\tau} \Bigl| 4 \pi i
     \sum_{n\in \ZZ} n\, A_n  \varphi'(t) e^{2\pi i n \,\varphi(t)}\Bigr|^2 \, \tfrac{\dt}{\varphi'(t)}.
\]
Consider $\beta(t) = t{-}s(t)$ with domain $t \in
[0,+\infty)$.
Clearly, $\beta(t)$ is strictly increasing and since
$\beta(0) = -1 < 0$, there exist a unique $t_0$ such that
$\beta(t_0) =0$.  Let $\tau_0 := t_0{+}s(t_0) = \gamma(0)$. Then, by
Lemma~\ref{lem:ONS},
\[
   \norm{u_x(0,t)}_{L_2(0,\tau_0,\frac{1}{\varphi'(t)})}^2  = 16 \pi^2 \sum_{n\in \ZZ} n^2  |A_n|^2 
\]
Clearly,
\[
    \tfrac{1}{M(t_0)} \norm{u_x(0,t)}_{L_2(0,\tau_0)}^2
\le \norm{u_x(0,t)}_{L_2(0,\tau_0,\frac{1}{\varphi'(t)})}^2 
\le \tfrac{1}{m(t_0)} \norm{u_x(0,t)}_{L_2(0,\tau_0)}^2.
\]
Combining this with Proposition~\ref{prop:norm-equivalence-f-g}, we
find our double inequality. From this is obvious that observation
times $\tau \ge \tau_0$ suffice.  On the other hand, if
$\tau< \tau_0$, $\norm{u_x(0,t)}_{L_2(0,\tau,\frac{1}{\varphi'(t)})}^2$ and
$\sum n^2 |A_n|^2$ cannot be comparable, which is easy to see by a
change of variables bringing it back the the standard trigonometric
orthonormal basis of $L_2(0, 1)$. This shows, again by
Proposition~\ref{prop:norm-equivalence-f-g}, that exact observation is
impossible.
\end{proof}

\begin{theorem}\label{thm:obs-on-right-bdry}
  For the solution $u$ given by (\ref{eq:solution-as-series}) to the
  moving boundary wave equation {\rm
    (\ref{eq:1D-wave-eq-linear-bdry})} the following double
  inequality holds:
\begin{equation}\label{eq:thm-obs-on-right-bdry}
                     C_1 \, \bignorm{(g,f)}_{H_0^1 \times L_2}^2
\quad \leq \quad    \int_0^{\gamma^{-1}(0)}   \bigl| u_x(s(t), t) \bigr|^2\dt     
\quad \leq \quad    C_2\,\bignorm{(g,f)}_{H_0^1 \times L_2}^2 
\end{equation}
where $C_1 =  \tfrac{m(0)}{2 M(0)(1+\norm{s'}_\infty)} (1{+}\tfrac{m(t_0)}{M(t_0)})^2$ and
$C_2 =  \tfrac{M(0)}{2m(0) (1-\norm{s'}_\infty)} (1{+}\tfrac{M(t_0)}{m(t_0)})^2 $.

In particular, with the observations $ M(t)\psi = \psi_x(s(t))$ the
problem {\rm (\ref{eq:1D-wave-eq-linear-bdry})} is exactly observable
in time $\tau$ if and only if $\tau \geq \gamma^{-1}(0)$.
\end{theorem}

\begin{proof} 
  Next we consider observation on the right boundary $x=s(t)$.  As in
  the proof of Theorem~\ref{thm:obs-on-left-bdry}, let $t_0$ be such
  that $\beta(t_0) = t_0{-}s(t_0) = 0$ and define $\tau_0 := \gamma^{-1}(0)$.
  Taking the derivative of $u(x,t)$ with respect to $t$ term by term,
  substituting $x = s(t)$ and exploiting
  (\ref{eq:functional-equation}) yields
\begin{equation}  \label{eq:representation-of-u_x}
  \begin{split}
 u_x(s(t),t) 
= & \; 2\pi i \sum_{n\in \ZZ} n\, A_n \Bigl(e^{2\pi i n \,\varphi(t+s(t))} \varphi'(t+s(t))) + e^{2\pi i n \,\varphi(t-s(t))}\varphi'(t-s(t))\Bigr)\\
= & \; 2\pi i \sum_{n\in \ZZ} \varphi'(t-s(t)) e^{2\pi i n \,\varphi(t-s(t))}    n\, A_n 
    \Bigl( 1 + \frac{ \varphi'(t+s(t))}{  \varphi'(t-s(t))} \Bigr)
  \end{split}
\end{equation}
Then
\begin{equation}\label{eq:bdry-obs-equivalence}
   (1+\tfrac{m(t_0)}{M(t_0)}) \le  \Bigl(1 + \frac{ \varphi'(t+s(t))}{\varphi'(t-s(t))}  \Bigr) \le (1+\tfrac{M(t_0)}{m(t_0)})
\end{equation}
Let $\omega(t) = \frac{1-s'(t) }{ \varphi'(t-s(t))}$. Then
\[
\bignorm{u_x(s(t), t)}_{L_2(0,\tau_0, \omega(t)\dt)}^2 \sim 4\pi^2
\int_{0}^{\tau_0} \Bigl| \sum_{n\in \ZZ} e^{2\pi i n
  \,\varphi(t-s(t))} n\, A_n \Bigr|^2 \varphi'(t{-}s(t)) (1{-}s'(t)) \dt
\]
where the equivalence comes from (\ref{eq:bdry-obs-equivalence}). We
make the change of variables $\xi = \varphi(t{-}s(t))$ and observe that
(\ref{eq:functional-equation}) gives an upper bound of the integral
to be $\varphi(\beta(\tau_0))) = 1 + \varphi(\beta(0))$. So
\[
  \bignorm{u_x(s(t), t)}_{L_2(0,\tau_0,\omega(t)\dt)}^2 
\sim  4\pi^2  \int_{\varphi(\beta(0))}^{\varphi(\beta(0)) + 1}  \Bigl|  \sum_{n\in \ZZ}  e^{2\pi i n \xi} n A_n\,  \Bigr|^2 \dxi
 = 4\pi^2 \sum_{n\in \ZZ} n^2 |A_n|^2
\]
We summarise:
\[
 4\pi^2 (1+\tfrac{m(t_0)}{M(t_0)})^2   \sum_{n\in \ZZ} n^2 |A_n|^2 
\; \leq \; 
 \bignorm{u_x(s(t), t)}_{L_2(0,\tau_0,\omega(t)\dt)}^2 
\; \leq \; 
 4\pi^2 (1+\tfrac{M(t_0)}{m(t_0)})^2  \sum_{n\in \ZZ} n^2 |A_n|^2 
\]
We conclude the proof observing that 
$\frac{1-\norm{s'}_\infty }{M(0)} \le  \omega(t) \le \frac{1+\norm{s'}_\infty }{m(0)}$
which allows to remove the weight function:
\[
 \tfrac{4\pi^2 m(0)}{1+\norm{s'}_\infty} (1+\tfrac{m(t_0)}{M(t_0)})^2   \sum_{n\in \ZZ} n^2 |A_n|^2 
\; \leq \; 
 \bignorm{u_x(s(t), t)}_{L_2(0,\tau_0)}^2 
\; \leq \; 
 \tfrac{4\pi^2 M(0)}{1-\norm{s'}_\infty} (1+\tfrac{M(t_0)}{m(t_0)})^2  \sum_{n\in \ZZ} n^2 |A_n|^2 
\]
We conclude using Proposition~\ref{prop:norm-equivalence-f-g}.
\end{proof}


\begin{minipage}{.6\linewidth}
  Let us finish this paragraph with a little observation. The
  optimal times for boundary observations given in
  Theorems~\ref{thm:obs-on-left-bdry} and \ref{thm:obs-on-right-bdry}
  are precisely the times where a characteristic emerging from the
  left (resp. right) boundary point $x=0$, resp. $x=1$ hit again the
  boundary curve, see the picture on the right.

  A second remark is that since $u(s(t), t) = 0$, taking derivative
  with respect to $t$ gives $s'(t) u_x( s(t), t) = - u_t(s(t), t)$. We
  may hence replace $u_x$ by $u_t$ in the inequality
  (\ref{eq:thm-obs-on-right-bdry}), at the only price to modify the
  constants by a factor $\norm{s'}_\infty$ .
\end{minipage}
\hspace*{.5cm}
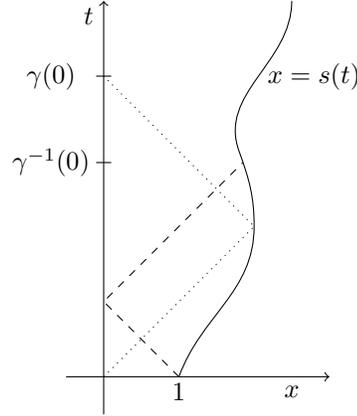
\begin{minipage}{.35\linewidth}
  \begin{tikzpicture}
  \draw[->] (-.5,0) -- (3,0) ;    \path (2.5, -.2) node {$x$} ; 
  \draw[->] (0,-.5) -- (0,5) ;    \path (-.2, 4.8) node {$t$} ; 
  \path (1, -.2) node {$1$} ;
  \draw (1, 0) to[out=70,in=270] (2,2)  to[out=90,in=290] (1.8, 3) to[out=110,in=270] (2.5,5) ; 
  \path (2.8, 4) node {$x=s(t)$} ;
  \draw[dotted] (0,0) -- (2, 2) -- (0, 4); 
  \draw (-.1,4) -- (.1, 4) ;   \path (-.7, 4) node {$\gamma(0)$} ;

  \draw[dashed] (1,0) -- (0, 1) -- (1.85, 2.85); 
  \draw (-.1, 2.85) -- (.1,2.85) ;   \path (-.7, 2.85) node {$\gamma^{-1}(0)$} ;
\end{tikzpicture}
\end{minipage}

Somehow a similar result to Theorem~\ref{thm:obs-on-right-bdry} in a
dual setting in terms of controllability have been shown in
\cite{Cui-Liu-Gao} for the special case of a linear moving wall
$s(t) = 1+\varepsilon t$ by a transformation to a cylindrical domain
proposed by Miranda \cite{Miranda}. The minimal control time estimate
was however far from optimal. Their result (again only for the linear
moving wall case) was subsequently improved in \cite{SunLiLu} who found
the same minimal control time as ourselves by a different
method\footnote{Caution: when writing out the parametrisation of the
  boundary integral in \cite[formula (2.2)]{SunLiLu}, the authors
  forget a factor $(1{+}\varepsilon)^{\nicefrac{1}{2}}$. This wrong
  factor then appears in many subsequent estimates in their paper.}.

\subsection{Internal Point observation}
Next, we turn our attention to observation on an internal point.  In
the situation where $s(t) = 1$ and hence $\varphi(x) = x$, the
solution $u$ to (\ref{eq:1D-wave-eq-linear-bdry}) is given by a
sine-series (due to Dirichlet boundary conditions),
\[
  u(x,t) = \sum_{n \in \ZZ} a_n e^{i\pi n t} \; \sin\bigl( n\pi x\bigr).
\]
Consequently, internal point observation at $x{=}a$ is not possible when
$a \in \QQ$ since then infinitely many terms in the sum vanish,
independently of the leading coefficient. One way to counter this
problem is to obtain observability results for the average of $|u|^2$
in a small neighbourhood of a fixed internal point $a$, see
\cite{FabrePuel}. It is also well known that another way to counter
this problem is to consider a moving interior point, see for example
\cite{Castro:moving-interior,Khapalov:1995,Khapalov:2001}.  We follow
in this article the idea that fixed domain with moving observers
should somehow behave similar to moving domains with fixed
observers. The following result confirms this intuition: for any fixed
point $a \in (0,1)$, consider a Neumann observer defined by
$Cu = u_t(a,t)$ to the solution $u$ of the moving boundary wave
equation (\ref{eq:1D-wave-eq-linear-bdry}).

\begin{theorem} \label{thm:obs-on-internal-point} Let $s$ be an
  monotonic admissible boundary curve and $\varphi$ be a
  $\Ce^2$-solution to (\ref{eq:functional-equation}). Assume
  additionally that $\varphi'$ is strictly decreasing if $s(\cdot)$ is
  increasing or that $\varphi'$ is strictly increasing if $s(\cdot)$ is
  decreasing, respectively.
 
  Then solution $u$ to the wave equation
  {\rm (\ref{eq:1D-wave-eq-linear-bdry}) } satisfies the following
  double inequality:
\[
C_1(a) \; \bignorm{(g,f)}_{H_0^1 \times L_2}^2
\quad \le \quad  
 \int_0^{a+\gamma(-a)} \bigl| u_x(a,t) \bigr|^2 \dt 
\quad \le \quad  
C_2(a) \; \bignorm{(g,f)}_{H_0^1 \times L_2}^2,
\]
where the constants $C_1$ and $C_2$ depend only on $\varepsilon$ and
$a$. We provide them explicitly in the proof. 
\end{theorem}
\begin{proof}
  Let $t_1 = \beta^{-1}(-a)$ and $\tau_a = a + \gamma(-a)$.
  Term by term differentiation of (\ref{eq:solution-as-series}) with
  respect to $t$ gives
\[
u_t(a,t)  
= 
2 \pi i \sum_{n\in \ZZ} n \, A_n \Bigl(e^{2\pi i n \,\varphi(t+a)}\varphi'(t+a) + e^{2\pi i n \,\varphi(t-a)}\varphi'(t-a)\Bigr) 
\]
First we suppose that $\varphi'$ is strictly decreasing.
We first calculate a weighted $L_2$-norm with $\omega_a(t) = \tfrac{1}{\varphi'(t-a)}$:
\[
A - B \quad \le \quad \norm{u_x(a,t)}_{L_2(0,\tau_a,\omega_a(t)\dt)} \quad \le \quad A + B
\]
with 
\begin{align*}
A :=  & \;  2\pi \Bignorm{ \sum_{n\in \ZZ} n \, A_n e^{2\pi i n \,\varphi(t-a)}\varphi'(t-a) }_{L_2(0,\tau_a,\omega_a(t)\dt)} \\
B := & \;   2\pi  \Bignorm{ \sum_{n\in \ZZ} n\, A_n e^{2\pi i n \,\varphi(t+a)}\varphi'(t+a) }_{L_2(0,\tau_a,\omega_a(t)\dt)}.
\end{align*}
To estimate $A$, the change of variables $s=t-a$ together with
Lemma~\ref{lem:ONS} therefore gives
\[ 
    A^2 = 4\pi^2 \sum_{n\in \ZZ} n^2|A_n|^2.
\]       
For $B$, we have
\[
B^2 =4\pi^2 \int_{0}^{\tau_a} \Bigl|\sum_{n\in \ZZ} n \, A_n(e^{2\pi i n \,\varphi(t+a)}\varphi'(t+a)) \Bigr|^2 \omega_a(t)\dt
\]
Since $\varphi'$ is strictly decreasing,
$0 < \tfrac{\varphi'(t+a)}{\varphi'(t-a)} < 1$ for all
$t \in [0, \tau_a]$ and so
$q_a := \max_{[0, \tau_a]} \tfrac{\varphi'(t+a)}{\varphi'(t-a)} < 1$.
We then have
\begin{align*}
B^2 \leq & \; 4\pi^2  q_a
         \int_{0}^{\tau_a} \Bigl|\sum_{n\in \ZZ} n\, A_n e^{2\pi i n \,\varphi(t+a)}\varphi'(t+a)) \Bigr|^2 \tfrac{1}{\varphi'(t+a)} \dt\\
= & \; 4\pi^2 q_a
\int_{a}^{a+\tau_a}  \Bigl|\sum_{n\in \ZZ} n\, A_n e^{2\pi i n \,\varphi(s)} \Bigr|^2 {\varphi'(s)} \ds
\end{align*}
Recall that $a+\tau_a = 2a + \gamma(-a)$. Since $s'\ge 0$, we have
$\gamma' \ge 1$ and so $2a + \gamma(-a) \le \gamma(a)$.  By
Lemma~\ref{lem:ONS} we infer
\[B^2 \leq  4\pi^2 q_a
\int_{a}^{\gamma(a)}  \Bigl|\sum_{n\in \ZZ} n\, A_n e^{2\pi i n \,\varphi(s)} \Bigr|^2 {\varphi'(s)} \ds =  4 \pi^2 q_a \sum_{n\in \ZZ} n^2|A_n|^2.
\]
Putting both on $A$ and $B$ estimates together, and using Proposition~\ref{prop:norm-equivalence-f-g}, we get the lower estimate
\begin{align*}
       \norm{u_t(a, t) }_{L_2(0, \tau_a)}^2
\geq & \; m(t_1)   \norm{u_t(a,t)}_{L_2(0,\tau_a,\omega_a(t)\dt)}^2 \\
\geq & \; 4\pi^2 m(t_1) (1{-}\sqrt{q_a})^2  \sum_{n\in \ZZ} n^2|A_n|^2 \\
\geq & \; C_1(a)   \; \bignorm{(g,f)}_{H_0^1 \times L_2 }^2
\end{align*}
with $c_1(a) = \tfrac{m(t_1)}{2M(0)} (1{-}\sqrt{q_a})^2$.
The upper estimate is similar; we find $C_2(a) =  \tfrac{M(t_1)}{2m(0)} (1{+}\sqrt{q_a})^2$.

In the case where $\varphi'$ is strictly increasing we use
$\widetilde{\omega_a}(t) = \tfrac{1}{\varphi'(t+a)}$ as a weight function and 
change the rôles of $A$ and $B$. The result follows the same lines then.
\end{proof}

We observe that the same proof also gives the double inequality
\[
C_1(a) \; \bignorm{(g,f)}_{H_0^1 \times L_2}^2
\quad \le \quad  
 \int_0^{a+\gamma(-a)} \bigl| u_t(a,t) \bigr|^2 \dt 
\quad \le \quad  
C_2(a) \; \bignorm{(g,f)}_{H_0^1 \times L_2}^2.
\]

\subsection*{Discussion} One may formulate
(\ref{eq:1D-wave-eq-linear-bdry}) as an abstract non-autonomous Cauchy
problem, for example as follows: let $H_t = L_2([0, s(t)])$ and define
\[
\DOMAIN(A(t)) = H^1_0([0, s(t)] \cap H^2([0, s(t)])
\qquad\text{and}\qquad
A(t) f = f''
\]
Then $A(t)$ is the generator of an analytic semigroup on $H_t$. For $t \ge 0$, 
we let $\calH_t = H_0^1([0, s(t)]) \times L_2([0, s(t)])$ and
\[
\DOMAIN(\calA(t)) = \DOMAIN(A(t)) \times  H_0^1([0, s(t)])
\qquad\text{and}\qquad
   \calA(t) = \Bigl( 
\begin{array}[rl]{cc}
  0 & I \\ A(t) & 0 
\end{array}\Bigr).
\]
With this notation (\ref{eq:1D-wave-eq-linear-bdry}) rewrites as
\begin{equation}  \label{eq:abstract-nonaut-pb}\left\{
  \begin{split}
        x'(t) = & \; \calA(t) x(t) \\
        x(0) = & \; x_0 = (g, f) \in \calH_0.  
  \end{split}\right.
\end{equation}
The observation of $t \mapsto u_x(a, t)$ discussed in the theorem is
then realised with observation operators
$C(t) : \DOMAIN(\calA(t)) \to \CC$ defined by
$C(t) ({v}, {w})^t = v_x(a)$.  Theorem~\ref{thm:obs-on-internal-point}
states in particular exact observability on $[0, \tau]$ if and only if
$\tau \ge a+\gamma(-a)$. It is remarkable that this holds true,
although, for a dense subset of values of $t_0$ (precisely if
$a/s(t_0) \in \QQ$) the ``frozen'' evolution equations
\[
   x'(t) + \calA(t_0) x(t) = 0 \qquad y(t) = C(t) x(t)
\]
are {\em not} exactly observable by the sine-series argument given
above for the case $s(t)=1$. This could now lead to the intuition that
the non-observability on for all $t>0$ such that $a/s(t) \in \QQ$ is
an ``almost everywhere phenomenon'', and may be ignored. This idea is
partially contradicted by the following result, where the 
observation position depends on time and may be such that the ratio
$a(t) / s(t) \in \QQ$ for all $t>0$.

\begin{theorem}\label{thm:obs-by-movin-interval-point}
  Let $s(t) = 1+\varepsilon t$ and $a(t) = a s(t)$ for some $a \in (0,1)$.
  Then the solution $u$ to the wave equation {\rm
    (\ref{eq:1D-wave-eq-linear-bdry}) } satisfies the following
  admissibility and observation inequality:
\[ 
C_1(a, \varepsilon) 
\; \bignorm{ (g,f)}_{H_0^1 \times L_2}^2 
\quad \le \quad
\int_0^{\frac{2}{1{-}\varepsilon} } \bigl| u_t(a(t), t) \bigr|^2  \dt
\quad \le \quad
C_2(a, \varepsilon) 
\; \bignorm{ (g,f)}_{H_0^1 \times L_2}^2 
\]
The constants $C_1$ and $C_2$ depend only on $a$ and $\varepsilon$.
We provide them explicitly in the proof. 
\end{theorem}
\begin{proof} 
  Recall that the solution $u$ of the equation
  (\ref{eq:1D-wave-eq-linear-bdry}) can be written in the form (\ref{eq:solution-as-series}).
  Taking the derivative respected to $t$ gives
\[
 u_t(x,t) = 2\pi i \sum_{n\in \ZZ}  n A_n  \,
   \Bigl(e^{2\pi i n \,\varphi(t{+}x)} \varphi'(t{+}x)- 
         e^{2\pi i n \,\varphi(t{-}x)}\varphi'(t{-}x)\Bigr)
\]
Substituting $x = a(t)$, we get 
\[
 u_t(a(t),t) = 2\pi i 
 \sum_{n\in \ZZ} n A_n\, \Bigl(
e^{2\pi i n \,\varphi(t+a(1{+}\varepsilon t))} \varphi'(t+a(1{+}\varepsilon t))  - 
e^{2\pi i n \,\varphi(t-a(1{+}\varepsilon t))}\varphi'(t-a(1{+}\varepsilon t))\Bigr)
\]
By calculation, we have the followings identities
\begin{align*}
  \varphi(t\pm a(1{+}\varepsilon t)) & = \; \varphi(t)+\varphi(\pm a)\\
\varphi_t(t\pm a(1{+}\varepsilon t)) &= \;\tfrac{1}{\varepsilon}\varphi'(t)\varphi'(\pm a)
\end{align*}
Plugging them into the preceding equation we get  
\begin{align*}
 u_t(a(t),t)\;  
& = \quad \tfrac{2\pi i}{\varepsilon}
\sum_{n\in \ZZ}A_n \Bigl(e^{2\pi i n \,(\varphi(t)+\varphi( a))} \varphi'(t)\varphi'( a)- e^{2\pi i n \,(\varphi(t)+\varphi(- a))}\varphi'(t)\varphi'(- a))\Bigr) \\
& = \quad \tfrac{2\pi i}{\varepsilon} \sum_{n\in \ZZ}A_n  e^{2\pi i n \,\varphi(t)}\varphi'(t)
      \Bigl(e^{2\pi i n \,\varphi( a)}\varphi'( a)- e^{2\pi i n \,\varphi(- a)}\varphi'(- a)\Bigr)
\end{align*}
Let $t_0 = \tfrac{1}{1{-}\varepsilon}$. Then $[t_0{-}s(t_0), t_0{+}s(t_0) = [0, \tfrac{2}{1{-}\varepsilon}]$ 
and so, using Lemma~\ref{lem:ONS}, 
\begin{align*}
  & \; \bignorm{u_t(a(t), t)}_{L_2(0,\frac{2}{1{-}\varepsilon},\frac{1}{\varphi'(t)})}^2 \\
= & \; \tfrac{4\pi^2}{\varepsilon^2 } \int_{0}^{\frac{2}{1{-}\varepsilon}}
       \biggl| \sum_{n\in \ZZ}
       e^{2\pi i n \,\varphi(t)}\varphi'(t) \,  n A_n   \Bigl(e^{2\pi i n \,\varphi( a)}\varphi'( a)- 
       e^{2\pi i n \,\varphi(- a)}\varphi'(- a)\Bigr)\biggr|^2 \tfrac{1}{\varphi'(t)} \dt \\
= & \; \tfrac{4 \pi^2}{\varepsilon^2 }\sum_{n\in \ZZ} n^2 |A_n|^2 
          \bigl|e^{2\pi i n \,\varphi( a)}\varphi'( a) - 
                e^{2\pi i n \,\varphi(- a)}\varphi'(- a))\bigr|^2 
\end{align*}
Now we need to estimate the multiplicative term
\begin{align*}
M_n^2 =  & \; \bigl|  e^{2\pi i n \,\varphi( a)}\varphi'( a)- e^{2\pi i n \,\varphi(- a)}\varphi'(- a)) \bigr|^2 \\
= & \;  \varphi'( a)^2+\varphi'(- a)^2 - 2\varphi'( a)\varphi'(- a)\cos\Bigl( 2\pi n(\varphi( a) - \varphi(- a))\Bigr).
\end{align*}
Clearly,
$(\varphi'( a)-\varphi'(- a))^2 \le M_n^2 \le
(\varphi'( a)+\varphi'(- a))^2$ ; by direct calculation,
\[
(\varphi'( a)-\varphi'(- a))^2 \; =  \; \eta_\varepsilon^{-2} \frac{4\varepsilon^4 a^2}{(1{-}\varepsilon^2 a^2)^2}
\quad\text{and}\quad
(\varphi'( a)+\varphi'(- a))^2 \; =  \; \eta_\varepsilon^{-2} \frac{4\varepsilon^2}{(1{-}\varepsilon^2 a^2)^2}
\]
Therefore, by Proposition~\ref{prop:norm-equivalence-f-g}, 
\[
     \tfrac{16\pi^2\varepsilon^2 a^2 }{(1{-}\varepsilon^2 a^2)^2 \eta_\varepsilon^2}  
               \sum_{n\in \ZZ} n^2 |A_n|^2
\quad\le\quad
     \norm{u_t(a(t), t)}_{L_2(0,\frac{2}{1{-}\varepsilon},\frac{1}{\varphi'(t)})}^2  
\quad\le \quad  
     \tfrac{16\pi^2}{(1{-}\varepsilon^2 a^2 )^2\eta_\varepsilon^2}
               \sum_{n\in \ZZ} n^2 |A_n|^2
\]
Now we apply Proposition~\ref{prop:norm-equivalence-f-g} to conclude.
We find
\[
C_1(a, \varepsilon) = \tfrac{1-\varepsilon}{1+\varepsilon} 
                      \tfrac{2 \varepsilon^2 a^2 }{(1{-}\varepsilon^2
  a^2)^2 \eta_\varepsilon^2} 
\quad\text{and}\quad
C_2(a, \varepsilon) = \tfrac{1+\varepsilon}{1-\varepsilon} 
                     \tfrac{2 }{(1{-}\varepsilon^2  a^2)^2 \eta_\varepsilon^2} .\qedhere
\]
\end{proof}

\bigskip

\subsection{Simultaneous exact observability}
A last result in this section concerns simultaneous exact observability : consider a
system of two coupled 1D wave equations, one of which has a fixed
boundary, and the second has the moving domain $0 \le x \le s(t)$ as
above.  Assume that we can observe only the combined force exerted by
the strings at the common endpoint
$\varphi(t) = u_x^{(1)}(0,t) + u_x^{(2)}(0,t)$, for $t \in [0,T]$. The
question is whether we can still exactly observe all initial data. Our
system is defined as
\begin{equation}
  \label{eq:coupled wave eq}\tag{$W_2$}
  \left\{
    \begin{array}{ll}
    u_{tt} - u_{xx} = 0  \qquad                        & (x, t) \in \Omega\\
    v_{tt} - v_{xx} = 0  \qquad                        & -1 \leq x \leq 0 \\
    u(0, t) = u(s(t), t) =  v(-1, t) =  v(0, t) = 0 \qquad     & t \ge 0\\
    u(x,0) = g(x)  , u_t(x,0) = f(x)               \qquad                & x \in [0,1]\\
    v(x,0) = \widetilde{g}(x)  , v_t(x,0) = \widetilde{f}(x)         \qquad                & x \in [{-}1, 0]
    \end{array}
  \right.
\end{equation}
\begin{theorem}\label{thm:simultaneous-obs}
  Let $s(\cdot)$ be an admissible boundary curve and assume additionally  that either
  \[
   \liminf_{t\to \infty} \gamma'(t) > 1 
      \qquad \text{or} \qquad
   \gamma'(t) = 1 + a x^{-\delta} + \LANDAUo(t^{-\delta}), \quad 0< \delta< 1, a>0.
\]
  Moreover assume that $\varphi'$ is bounded on $\RR_+$.
  Let $(u, v)$ be the solution to {\rm (\ref{eq:coupled wave
    eq})}. Then, for all $\lambda>0$ there exists
  $\tau_0>2$ such that for all   $\tau \geq \tau_0$
\begin{equation}
  \lambda \Bigl( \bignorm{ (g, f)}_{H_{1}^0 \times L_1}^2 
              + \bignorm{( \widetilde{g}, \widetilde{f} ) }_{H_{1}^0 \times L_2}^2  \Bigr) 
\quad \leq \quad 
\int_{0}^{\tau} \bigl|u_x(0,t) + v_x(0,t)  \bigr|^2 dt
\end{equation}
\end{theorem}

Our assumptions include the cases of linear moving boundaries,
parabolic boundaries and hyperbolic boundaries.  However, for the
shrinking domain they are not satisfied.

\begin{proof} 
By the triangle inequality we have
\[
  \Bigl(\int_{0}^{\tau}\bigl|u_x(0,t) + v_x(0,t)  \bigr|^2 \dt\Bigr)^{\nicefrac{1}{2}} \geq A(\tau) - B(\tau)
\]
where 
\[
A(\tau)  = \Bigl(\int_{0}^{\tau}\bigl|v_x(0,t) \bigr|^2 \dt\Bigr)^{\nicefrac{1}{2}}
\quad \text{and} \quad
B(\tau) = \Bigl(\int_{0}^{\tau}\bigl|u_x(0,t) \bigr|^2 \dt\Bigr)^{\nicefrac{1}{2}}
\]
It is well known that the solution $v$ of the wave equation with
the fixed boundary can be expressed as a pure sine series
\begin{equation}\label{eq: express}
v(x,t) = \sum_{n \in \ZZ} a_ne^{\pi i n\, t} \; \sin\bigl( n\pi x\bigr),
\end{equation}
where $(n a_n )_{n\in \ZZ} \in \ell_2$ and hence
$(a_n)_{n\in\ZZ}, \in \ell_2$.  Consequently, for all $t \geq 0$, the
energy of $v$ is constant: indeed, by direct computation,
\[
E_{v}(t) = \tfrac{1}{2}\int_{0}^{1}\bigl(\tfrac{\partial v(x,t)}{\partial t}\bigr)^2 + \bigl(\tfrac{\partial v(x,t)}{\partial x}\bigr)^2 \dx 
             = \pi^2 \sum_{n\in \ZZ} n^2 a_n^2
\]
We also have
\[
\int_{0}^{2} \bigl|v_x(0,t) \bigr|^2 \dt
=   \; \int_{0}^{2} \Bigl|\sum_{n\in \ZZ} \pi n  a_n e^{i\pi n t} \cos\bigl( n\pi x\bigr)\Bigr|^2 \dt \\
=  \;  E(v)(0). 
\]
Hence, using periodicity of $v$, we obtain (recall $\tau\ge 2$) 
\[ 
A(\tau)^2 = \int_{0}^{\tau}\bigl|v_x(0,t) \bigr|^2 \dt \;\geq\;  \floor{\tfrac{\tau}{2}} \; E_{v}(0) 
\]
Next we turn to an estimate for $B(\tau)$. Recall that
\[
   u_x(0,t) = 2\pi i \sum_{n \in \ZZ}  n A_n  \varphi'(t) e^{ 2\pi i n \,  \varphi(t)}
\]
Let $t_0 = 0$ and $t_n = \gamma^{(n)}(t_0)$. By construction of $t_n$ and (\ref{eq:functional-equation}),
 \[
   \varphi(t_{n+1}) - \varphi(t_n) = \varphi( \gamma(t_{n})) - \varphi(t_n) =1.
 \]
 Hence, by Lemma~\ref{lem:ONS}, $e^{2\pi i n\,  \varphi(x)}$ is an
 orthonormal system on  $L_2([t_n, t_{n+1}], \varphi'(t)\dt)$.

 An inspection of the proof of Theorems~\ref{thm:abel-eq-1} and
 \ref{thm:abel-eq-2} shows that if
 $\liminf_{t\to \infty} \gamma' > 1$, $t_n \to + \infty$
 exponentially, whereas the asymptotics
 $ \gamma'(t) = 1 + a t^{-\delta} + \LANDAUo(t^{-\delta})$ ensures
 $t_n \sim c n^{\nicefrac{1}{\delta}}$. Let $N(\tau)$ be the unique
 integer satisfying $t_n \le \tau < t_{n+1}$.  Let
 $C = \sup\{ \varphi'(t) \SUCHTHAT t \ge 0 \}$.  Then
\begin{align*}
B(\tau) = \int_{0}^{\tau}\bigl|u_x(0,t) \bigr|^2 \dt
\leq  & \;  \int_{0}^{\tau}\bigl|u_x(0,t) \bigr|^2 \tfrac{1}{\varphi' (t)}\dt \\
\leq  & \; C \sum_{j=0}^{N(\tau)} \int_{t_j}^{t_{j+1}} \bigl|u_x(0,t) \bigr|^2 \tfrac{1}{\varphi' (t)}\dt \\
\leq  & \; 16\pi^2 C (N(\tau){+}1)  \sum_{n\in \ZZ} n^2 |A_n|^2 \\
\leq & \;  \tfrac{2C}{m(0)} (N(\tau){+}1)  \Bigl( \norm{g^{(1)}(x)}_{H_{1}^0(0,1)}^2 + \norm{f^{(1)}(x)}_{L_2(0,1)}^2 \Bigr).
\end{align*}
We obtained so far that 
\begin{align*}
     & \; \int_{0}^{\tau}\bigl|u_x(0,t) + v_x(0,t)  \bigr|^2 \dt \; 
\geq  \; A(\tau)^2 - B(\tau)^2 \\
\geq & \; \lfloor \tfrac{\tau}{2} \rfloor\; E_{v}(0)   -  \tfrac{2C}{m(0)} (N(\tau){+}1)  \Bigl( \norm{g^{(1)}(x)}_{H_{1}^0(0,1)}^2 + \norm{f^{(1)}(x)}_{L_2(0,1)}^2 \Bigr) 
\end{align*}
The first term grows linearly in $\tau$. The second term is
$\LANDAUo(\tau)$ since in case of exponential growth of the sequence
$t_n$, $N(\tau)$ behaves logarithmically and in case that
$t_n \sim c n^{\nicefrac{1}{\delta}}$, $N(\tau) \sim \tau^\delta$ with
$\delta < 1$.  Hence, the difference tends to infinity with
$\tau\to +\infty$, which means that for all $\lambda>0$ there exists
$\tau_0>0$ such that for $\tau \ge \tau_0$,
\begin{align*}
 \int_{0}^{\tau}\bigl|u_x(0,t) + v_x(0,t)  \bigr|^2 \dt \; 
\ge  & \; 2 \lambda\bigl( E(u)(0) +  E_{v}(0) \bigr)\\
=    & \; \lambda \Bigl( \bignorm{ (g, f)  }_{H_0^1\times L_2}^2  +  \bignorm{ (\widetilde{g}, \widetilde{f})  }_{H_0^1\times L_2}^2 \Bigr).\qedhere
\end{align*}
\end{proof}

\bigskip

\subsection{Duality results}\label{sec:duality}

Without detailed proofs we state dual results to our results
formulated as null-controllability in the sense of 'transposition'.

\subsubsection*{Dirichlet control on boundary} 
Let $s$ be an admissible boundary curve, $v$ the solution to the wave
equation on $\Omega$. Let $(G v)(t) = (v(0, t), v(s(t), t))$ be the
trace of $v$ on the two boundary points.  Then for either choice,
$\zeta(t) = (y(t), 0)$ or $\zeta(t) = (0, y(t))$ the boundary
controlled wave equation
\begin{equation}
  \label{eq:dual-problem-left}
  \left\{
    \begin{array}{rll}
      v_{tt} - v_{xx} & = 0  \qquad                        & (x, t) \in \Omega\\
      (G v)(t) & = \zeta(t)  \qquad     & t \ge 0\\
      v(x,0) & = g \in L_2([0,1])  \qquad       & x \in [0,1]\\
      v_t(x, 0) & = f \in H^{-1}([0,1]) \qquad       & x \in [0,1]
    \end{array}
  \right.
\end{equation}
is null-controllable in times $\tau = \gamma(0)$ in case  $\zeta(t) = (y(t), 0)$
and in time $\tau=\gamma^{-1}(0)$ in case $\zeta(t) = (0, y(t))$.
The null control can be achieved by the control function
$y(t) = {-}u_x(0, t)$, or $y(t) = {-}u_x(s(t), t)$, respectively
where $u(\cdot)$ is the solution to (\ref{eq:1D-wave-eq-linear-bdry}).

\subsubsection*{Simultaneous Null Control} 
Next we focus on the dual statement to
Theorem~\ref{thm:obs-on-internal-point} in terms of
null-controllability.  Instead of one wave equation on $\Omega$, we
consider two wave equations with mixed boundary conditions, one on the
cylindrical domain $[0, a] \times \RR_+$ and one on the
non-cylindrical domain $\{ (x, t) \SUCHTHAT a \le x \le s(t) \}$. Both
equations are coupled via the control function $\zeta$ in the
following way:
\begin{equation}\label{eq:coupled-control}
  \left\{
    \begin{array}{ll}
    v_{tt} - v_{xx} = 0  \qquad                        & 0 \leq x \leq a \\
    w_{tt} - w_{xx} = 0  \qquad                        & a \le x \le s(t) \\
    v(0, t) = w(s(t), t) = 0                         & t \ge 0 \\
    v(a{-}, t) = w(a{+}, t)                          & t \ge 0\\
    v_x(a{-}, t) - w_x(a{+}, t) = \zeta(t)           & t \ge 0\\
    v(x,0) = g(x), \quad v_t(x,0) = f(x)             & x \in [0,a]\\
    w(x,0) = g(x), \quad w_t(x,0) = f(x)             & x \in [a, 1]
    \end{array}
  \right.
\end{equation}
Then Theorem~\ref{thm:obs-on-internal-point} implies that
(\ref{eq:coupled-control}) is null-controllable in time
$\tau \ge a{+}\gamma({-}a)$. The control can be achieved by letting
$\zeta(t) = u_x(a, t)$ where $u(\cdot)$ is the solution to
(\ref{eq:1D-wave-eq-linear-bdry}).

\vspace*{2cm}

\appendix

\section{Differentiable solutions for general boundary functions}\label{sec:appendix}

\newcommand{\snull}{1}
\newcommand{\gammaofminusone}{1}

In this section we discuss the solvability of
(\ref{eq:functional-equation}) by a differentiable function $\varphi$.
Our hypotheses are that the boundary function $s$ be of class $\Ce^1$
at least and that $\lim_{t \to \infty} s'(t) = \cals$ exists. This
last condition is of course only of interest if we seek for solutions
$\varphi$ satisfying (\ref{eq:functional-equation}) for $t \in \RR_+$,
since it can easily be arranged if we consider only $t \in [0, \tau]$.

Let $s(\cdot)$ be of class $\Ce^1$ and $\norm{ s'}_\infty < 1$.  Let
$\alpha(t) = t + s(t)$ and $\beta(t) = t - s(t)$. Both functions,
$\alpha$ and $\beta$ are strictly increasing and continuous. Moreover,
$\alpha(t) = \alpha(0) + t \alpha'(\xi_t) > \alpha(0) + t
(1-\norm{s'}_\infty)$
yields $\lim_{t\to +\infty} \alpha(t) = +\infty$. Hence $\alpha$ is a
bijection from $[0, \infty)$ to $[{\snull}, \infty)$; similarly
$\beta$ is a bijection from $[0, \infty)$ to $[-{\snull}, \infty)$.
We then consider the bijection
\[
  \gamma := \alpha \circ \beta^{-1}: [{-}{\snull}, \infty) \to [{+}{\snull}, \infty).
\]
Observe that
\[
\gamma'(t) = \frac{\alpha' \circ \beta^{-1}}{\beta' \circ \beta^{-1}} 
           = \frac{ 1 + s'(\beta^{-1}(t)) } { 1 - s'(\beta^{-1}(t)) },
\]
so that $\gamma$ is strictly increasing by $\norm{s'}_\infty < 1$. The
sign of $s'(\beta^{-1}(t))$ determines whether $\gamma$ is strictly
contractive or strictly expansive. We also note for further reference
that if $s \in \Ce^2$,
\[
\gamma''(t) = \frac{2 s''(\beta^{-1}(t)) }{(1-s'(\beta^{-1}(t)))^3}.
\]
The functional equation (\ref{eq:functional-equation}) can now be
rephrased as
\begin{equation}  \label{eq:abels-equation}\tag{A}
   \varphi\circ \gamma =  \varphi + 1.
\end{equation}
This equation is known as 'Abel's equation' and intensively studied,
see for example \cite{Kuczma,KuczmaChoczewskiGer} and references therein.

We will consider only the case where $\lim s(t) = \cals$ exists.
Since $s(t) > 0 $ for all $t$, $\lim s(t) = \cals < 0$ is
impossible. We may therefore either have $\cals= 0$ or
$\cals \in (0,1)$.  We first discuss the situation of a non-zero
limit, which means that
$\gamma'(t) \to \ell = \tfrac{1+\cals}{1-\cals} > 1$.

\begin{theorem}\label{thm:abel-eq-1}
  Let $\ell> 1$ and assume that
  $\gamma'(x) = \ell + \LANDAUO(x^{-\delta})$ for $\delta>0$. Then
  Abel's equation (\ref{eq:abels-equation}) admits a strictly
  increasing solution $\varphi \in \Ce^1([{-}{\snull}, \infty))$.  If
  additionally $\gamma \in \Ce^2[0, \infty)$,
  $\gamma'' = \LANDAUO(x^{-1-\delta})$ and $\gamma'$ is decreasing,
  then $\varphi$ is of class $\Ce^2([{-}{\snull}, \infty))$.
\end{theorem}

\begin{proof}[Proof of Theorem~\ref{thm:abel-eq-1}]
  Put $\psi = \ell^\varphi$. Then $\psi$ satisfies the Schröder equation
  $\psi \circ \gamma = \ell \psi$.  Since $\gamma({-}{\snull})={+}{\snull}$ and $\gamma$
  has no fixed points (otherwise $s(t)=0$), $\gamma(x) > x $ for all
  $x\ge {-}{\snull}$.  Observe that by assumption, there exists some $\xi> 0$
  such that $\gamma'(x) \ge \tfrac{1+\ell}2 > 1$ for all $x \ge \xi$. Let
  $a_0 = {-}{\snull}$ and $a_n = \gamma^{(n)}(a_0)$. If $(a_n)$ were bounded,
  we could extract a subsequence that converges to a fixed point of
  $\gamma$. So $a_n \to \infty$. Let $k$ be such that
  $a_k > \xi$. Hence
\[
a_{n+k+1} - \xi 
\; \ge \; \gamma(a_{n+k}) - \gamma(\xi) 
\;  >  \; \tfrac{1+\ell}2 (g_{n+k}- \xi) 
\]
shows that $a_n \to + \infty$ exponentially. By monotonicity of
$\gamma$ we infer the same for $\gamma^{(n)}(x) \ge a_n$
for all $x \ge {-}{\snull}$. This, together with
$\gamma'(x) = \ell + \LANDAUO(x^{-\delta})$ shows that
\[
    P(x) = \prod_{n=0}^\infty \frac{ \gamma'( \gamma^{(n)}(x) ) }{\ell}
\]
converges absolutely and uniformly on $[{-}{\snull}, \infty)$. $P$ vanishes
nowhere and satisfies  $P \circ \gamma = \frac{\ell}{\gamma'} P$. We define
\[
   \psi(x) := \int_{{\snull}}^x P(t)\,dt + C
\]
where the constant $C$ is to be determined. By construction, $\psi$ is
strictly increasing and satisfies
\[
   \psi \circ \gamma(x) 
  = \int_{\gamma({-}{\snull})}^{\gamma(x)} P(t)\,dt + C
  = \ell \int_{{-}{\snull}}^x P(t)\,dt + C
  = \ell \int_{{-}{\snull}}^{{\snull}} P(t)\,dt  + \ell \psi + C (1-\ell) 
\]
So that, letting $C = \frac{\ell}{\ell-1} \int_{{-}{\snull}}^{{\snull}} P(t)\,dt > 0$
ensures $\psi \circ \gamma = \ell \psi $ as required.  Then
$\varphi := \frac{\ln \psi}{\ln(\ell)}$ is of class $\Ce^1$, strictly
increasing.

If additionally $\gamma'$ decreases towards $\ell$ at infinity, a new
lecture of the above growth rate of $(x_n)$ shows that
$\limsup \tfrac{\ell^n}{x_n} \le 1$ for any $x_0 \ge {-}{\snull}$.  Therefore,
the (termwise differentiated product $P$) yields a series
\[
\sum_n \gamma''(x_n) \Bigl(\prod_{j=0}^{n-1} \gamma'(x_j)\Bigr)  \Bigl(\prod_{k\not=n}   \frac{ \gamma'(x_n ) }{\ell}\Bigr)
\]
that normally on $[{-}{\snull}, \infty)$. We infer that $P$ is of class
$\Ce^1$, hence $\psi$ and $\varphi$ of class $\Ce^2$.
\end{proof}

In the situation that $\lim s'(t) = \cals = 0$ and hence
$\lim \gamma'(t) = 1$ things are more delicate. If 
$\gamma$ is such that $\gamma'(x) = 1 + \LANDAUo(x^{-\delta})$ at infinity,
for all $x, y$, 
\[
\lim_{n \to \infty}  \frac{ \gamma^{(n+1)}(x) - \gamma^{(n)}(x) }{ \gamma^{(n+1)}(y) - \gamma^{(n)}(y) }
= 1.
\]
We leave the proof as exercise, as it is a modification of \cite[Lemma 7.3]{Kuczma}. Consequently, whenever
\[
\varphi(x) := \lim_{n \to \infty}  \frac{ \gamma^{(n)}(x) - \gamma^{(n)}(x_0) }{ \gamma^{(n+1)}(x_0) - \gamma^{(n)}(x_0) }
\]
exists, $\varphi$ is a solution to Abel's equation (\ref{eq:abels-equation}). This is
the P. Lévy's algorithm, see e.g.  \cite[Chapter VII]{Kuczma}.  In
order to ensure existence of a solution we will in general have to get
a finer control of the asymptotics. The next result in this direction
is based on ideas of Szekeres \cite[Theorem~1c]{Szekeres.acta.1958},
see also \cite[Theorem~7.2]{Kuczma}). The principal idea is similar to
Theorem~\ref{thm:abel-eq-1}, but we have to transform differently and 
to be more careful how to construct an infinite product.

\begin{theorem}\label{thm:abel-eq-2}
  If  $\gamma'(x) = 1 + a (1-\delta)x^{-\delta} + \LANDAUo(x^{-\delta})$
  at infinity, where $a>0$ and $\delta > 0$, $\delta\not=1$, then Abel's equation
  (\ref{eq:abels-equation}) has a strictly positive and strictly
  increasing $\Ce^1$-solution $\varphi$.
\end{theorem}
\begin{proof}
  First observe that
  $\frac{\gamma(x)}{x} = 1 + a x^{-\delta} + \LANDAUo(x^{-\delta})$, by integrating $\gamma'$ on $[0, x]$ or $[x, \infty)$ according to $\delta< 1$ or $\delta > 1$.
%
%
  First we transform our problem into a multiplicative version. To
  this end, let $g: [{-}{\snull}, \infty) \to (0, \infty)$ be a
  $\Ce^1$-function.  Then, whenever $\varphi$ solves Abel's equation
  (\ref{eq:abels-equation}), $\psi(x) = g(x) \varphi'(x)$ satisfies
\[
   (\psi\circ \gamma)(x)  
=   g(\gamma(x))  \varphi'(\gamma(x)) 
=   g(\gamma(x)) \frac{\varphi'(x)}{\gamma'(x)}
=   \frac{g(\gamma(x))}{g(x) \gamma'(x)}   \psi(x) = : m(x) \psi(x)
\]
Let $x_n = \gamma^{(n)}(x)$. If $(x_n)$ were bounded, it would
converge to a fixed point of $\gamma$ --- but there is none. So
$x_n \to + \infty$. Assume that we chose the function $g$ such that
\begin{equation}
  \label{eq:sum-test}
   \sum_n \left| \frac{g(x_n) \gamma'(x_n)}{g(x_{n+1})}  - 1 \right|
\end{equation}
converges uniformly on compact intervals. Then the infinite product
\begin{equation}  \label{eq:infinite-product}
   P(x)  
=  \prod_{n=0}^\infty \tfrac{1}{m (\gamma^{(n)}(x) )}
=  \prod_{n=0}^\infty  \frac{g(x_n) \gamma'(x_n)}{g(x_{n+1})} ,
\end{equation}
defines a continuous function $P$ that solves
$\psi\circ \gamma = m \cdot \psi$.  From $P$ we then easily regain
$\varphi$. We chose $g(x) = \gamma(x)^{1-\delta}$. Then $P(x)>0$ for
all $x$. Moreover we have the following asymptotics for
$x \to \infty$:
\begin{align*}
1 \; - \; \gamma'(x) \left( \frac{x}{\gamma(x)} \right)^{1{-}\delta} 
= & \; 
 1 \; - \; \tfrac{1}{\left(1 + a x^{-\delta} + r_1(x) \right)^{1{-}\delta} }
 \left(1 + a (1{-}\delta)x^{-\delta} + \widetilde{r_1}(x) \right) \\
= & \; 
1 \; - \; \left(1 - a (1{-}\delta) x^{-\delta} + r_2(x)\right) \left( 1 + a(1{-}\delta) x^{-\delta} + \widetilde{r_2}(x)\right) \\
= & a^2 (1{-}\delta)^2 x^{-2\delta} + r(x).
\end{align*}
where
$r_1, r_2, \widetilde{r_1} \widetilde{r_2} = \LANDAUo(x^{-\delta})$
and $r = \LANDAUo(x^{-2\delta})$ for $x \to \infty$. Next, we need a
growth rate for the orbits $x_n = \gamma^{(n)}(x_0)$: Observe that
$a = \lim_{n \to \infty} \frac{\gamma(x_n) - x_n}{x_n^{1-\delta}} = \lim_{n\to \infty }\frac{ x_{n+1} - x_n }{x_n^{1-\delta}}$. Rewriting the right hand side we obtain
\[
a = \lim_{n\to \infty}  (x_n^\delta - x_{n+1}^\delta)\left( \tfrac{x_{n+1}}{x_n} \right)^{-\delta}  \frac{ \frac{x_{n+1}}{x_n} - 1 } { \left(\frac{x_{n+1}}{x_n}\right)^{-\delta} - 1 }.
\]
Using $\frac{x_{n+1}}{x_n} = \frac{ \gamma(x_n) }{x_n} \to 1$ as
$n \to \infty$ the last fraction has limit $-\nicefrac{1}{\delta}$ and
we obtain
\[
\delta a = \lim_{n\to \infty}  (x_{n+1}^\delta - x_n^\delta).
\]
Taking Cesaro sums, 
\[
\delta a = \lim_{n\to \infty} \frac1n \sum_{j=0}^{n-1}
(x_{j+1}^\delta- x_j^\delta) = \lim_{n\to \infty} \frac1n x_n^\delta.
\]
We infer finally $x_n \sim c \, n^{\nicefrac{1}{\delta}}$ when $n \to \infty$.
Putting both parts together, 
\[
\left| \frac{g(x_n) \gamma'(x_n)}{g(x_{n+1})}  - 1 \right|
=  a^2 (1{-}\delta)^2 x_n^{-2\delta} + r(x_n)
=  a^2 (1{-}\delta)^2 n^{-2} + r(x_n)
\]
where $r(x_n) = \LANDAUo(n^{-2})$. Therefore (\ref{eq:sum-test})
converges absolutely and uniformly on compact intervals so that
(\ref{eq:infinite-product}) converges to a strictly positive function
$P$. For $C>0$ to be determined in a moment, we let
\[
   \varphi(x) := C \int_{\gammaofminusone}^x \frac{P(t)}{\gamma(t)^{1-\delta}}\,dt.
\]
$P$ and $\gamma$ being strictly positive, $\varphi$ is positive,
strictly increasing and of class $\Ce^1$. Moreover,
\begin{align*}
\varphi(\gamma(x))  =  & \; C \int_{\gamma(-\snull)}^{\gamma(x)} \frac{P(t)}{\gamma(t)^{1-\delta}}\,dt \;  
  =  \; C \int_{{-}{\snull}}^{x} \frac{P(\gamma(s))}{\gamma(\gamma(s))^{1-\delta}} \gamma'(s)\,ds \\
 = & \; C \int_{{-}{\snull}}^{x} \frac{P(s) m(s) }{\gamma(\gamma(s))^{1-\delta}} \gamma'(s)\,ds \; 
  =  \; C \int_{{-}{\snull}}^{x} \frac{P(t)}{\gamma(t)^{1-\delta}}\,dt\\
 = & \; \varphi(x) + C  \int_{{-}{\snull}}^{\gammaofminusone} \frac{P(t)}{\gamma(t)^{1-\delta}}\,dt,
\end{align*}
so that adjusting $C$ (the integral being strictly positive) we obtain a
solution of Abel's equation (\ref{eq:abels-equation}).
\end{proof}

\subsubsection*{Acknowledgement} Both authors are indebted to Marius
Tucsnak for suggesting questions that lead us to find
Theorem~\ref{thm:obs-on-internal-point}.

\def\SUBMITTED{Submitted}
\def\TOAPPEAR{To appear in }
\def\PREPARATION{In preparation }

\nocite{*}

\providecommand{\bysame}{\leavevmode\hbox to3em{\hrulefill}\thinspace}


\begin{thebibliography}{10}

\bibitem{AmmariBchatniaMufti}
Ka{\"{\i}}s Ammari, Ahmed Bchatnia, and Karim El~Mufti, \emph{Stabilization of
  the nonlinear damped wave equation via linear weak observability}, NoDEA
  Nonlinear Differential Equations Appl. \textbf{23} (2016), no.~2, Art. 6, 18.

\bibitem{Balazs}
Nandor~L Balazs, \emph{On the solution of the wave equation with moving
  boundaries}, Journal of Mathematical Analysis and Applications \textbf{3}
  (1961), no.~3, 472 -- 484.

\bibitem{BardosChen}
Claude Bardos and Goong Chen, \emph{Control and stabilization for the wave
  equation. {III}. {D}omain with moving boundary}, SIAM J. Control Optim.
  \textbf{19} (1981), no.~1, 123--138.

\bibitem{BardosLebeauRauch}
Claude Bardos, Gilles Lebeau, and Jeffrey Rauch, \emph{Sharp sufficient
  conditions for the observation, control, and stabilization of waves from the
  boundary}, SIAM J. Control Optim. \textbf{30} (1992), no.~5, 1024--1065.

\bibitem{CannarsaDePratoZolesio}
Piermarco Cannarsa, Giuseppe Da~Prato, and Jean-Paul Zol{\'e}sio,
  \emph{Evolution equations in noncylindrical domains}, Atti Accad. Naz. Lincei
  Rend. Cl. Sci. Fis. Mat. Natur. (8) \textbf{83} (1989), 73--77 (1990).

\bibitem{Castro:rapidly}
C.~Castro, \emph{Boundary controllability of the one-dimensional wave equation
  with rapidly oscillating density}, Asymptot. Anal. \textbf{20} (1999),
  no.~3-4, 317--350.

\bibitem{Castro:moving-interior}
Carlos Castro, \emph{Exact controllability of the 1-{D} wave equation from a
  moving interior point}, ESAIM Control Optim. Calc. Var. \textbf{19} (2013),
  no.~1, 301--316.

\bibitem{CastroCindeaMunch}
Carlos Castro, Nicolae C{\^{\i}}ndea, and Arnaud M{\"u}nch,
  \emph{Controllability of the linear one-dimensional wave equation with inner
  moving forces}, SIAM J. Control Optim. \textbf{52} (2014), no.~6, 4027--4056.

\bibitem{CooperStrauss}
Jeffery Cooper and Walter~A. Strauss, \emph{Energy boundedness and decay of
  waves reflecting off a moving obstacle}, Indiana Univ. Math. J. \textbf{25}
  (1976), no.~7, 671--690.

\bibitem{Corbett:thesis}
Norman~C. Corbett, \emph{Initial moving-boundary value problems associated with
  the wave equation}, Ph.D. thesis, University of Manitoba, 1991.

\bibitem{Corbett:symmerty}
\bysame, \emph{A symmetry approach to an initial moving boundary value problem
  associated with the wave equation}, Can. Appl. Math. Q. \textbf{18} (2010),
  no.~4, 351--360.

\bibitem{Cui-Liu-Gao}
Lizhi Cui, Xu~Liu, and Hang Gao, \emph{Exact controllability for a
  one-dimensional wave equation in non-cylindrical domains}, Journal of
  Mathematical Analysis and Applications \textbf{402} (2013), no.~2, 612 --
  625.

\bibitem{Dodonov}
Viktor Dodonov, \emph{Modern nonlinear optics, part 1}, 2nd edition ed., vol.
  119, ch.~Nonstationary {C}asimir effect and analytical solutions for quantum
  fields in cavities with moving boundaries, pp.~309--394, Wiley,New York,
  2002.

\bibitem{EngelNagel}
Klaus-Jochen Engel and Rainer Nagel, \emph{One-parameter semigroups for linear
  evolution equations}, Graduate Texts in Mathematics, vol. 194,
  Springer-Verlag, New York, 2000, With contributions by S. Brendle, M.
  Campiti, T. Hahn, G. Metafune, G. Nickel, D. Pallara, C. Perazzoli, A.
  Rhandi, S. Romanelli and R. Schnaubelt.

\bibitem{FabrePuel}
Caroline Fabre and Jean-Pierre Puel, \emph{Pointwise controllability as limit
  of internal controllability for the wave equation in one space dimension},
  Portugal. Math. \textbf{51} (1994), no.~3, 335--350.

\bibitem{Gaffour}
L.~Gaffour, \emph{Analytical method for solving the one-dimensional wave
  equation with moving boundary}, Journal of Electromagnetic Waves and
  Applications \textbf{12} (1998), 1429--1430.

\bibitem{GaffourGrigorian}
L.~Gaffour and G.~Grigorian, \emph{Circular waveguide of moving boundary},
  Journal of Electromagnetic Waves and Applications \textbf{10} (1996), no.~1,
  97--108.

\bibitem{JaffardTucsnakZuazua}
St{\'e}phane Jaffard, Marius Tucsnak, and Enrique Zuazua, \emph{Singular
  internal stabilization of the wave equation}, J. Differential Equations
  \textbf{145} (1998), no.~1, 184--215.

\bibitem{Khapalov:2001}
A.~Y. Khapalov, \emph{Observability and stabilization of the vibrating string
  equipped with bouncing point sensors and actuators}, Math. Methods Appl. Sci.
  \textbf{24} (2001), no.~14, 1055--1072.

\bibitem{Khapalov:1995}
A.~Yu. Khapalov, \emph{Controllability of the wave equation with moving point
  control}, Appl. Math. Optim. \textbf{31} (1995), no.~2, 155--175.

\bibitem{Kuczma}
Marek Kuczma, \emph{Functional equations in a single variable}, Monografie
  Matematyczne, Tom 46, Pa\'nstwowe Wydawnictwo Naukowe, Warsaw, 1968.

\bibitem{KuczmaChoczewskiGer}
Marek Kuczma, Bogdan Choczewski, and Roman Ger, \emph{Iterative functional
  equations}, Encyclopedia of Mathematics and its Applications, vol.~32,
  Cambridge University Press, Cambridge, 1990.

\bibitem{Lions}
J.~L. Lions, \emph{Control and estimation in distributed parameter systems
  (frontiers in applied mathematics)}, ch.~Pointwise Control for Distributed
  Systems, Society for Industrial and Applied Mathematics, 1987.

\bibitem{Lions:}
J.-L. Lions, \emph{Contr\^olabilit\'e exacte, perturbations et stabilisation de
  syst\`emes distribu\'es. {T}ome 1}, Recherches en Math\'ematiques
  Appliqu\'ees [Research in Applied Mathematics], vol.~8, Masson, Paris, 1988,
  Contr{\^o}labilit{\'e} exacte. [Exact controllability], With appendices by E.
  Zuazua, C. Bardos, G. Lebeau and J. Rauch.

\bibitem{LuLiChenYao}
Liqing Lu, Shengjia Li, Goong Chen, and Pengfei Yao, \emph{Control and
  stabilization for the wave equation with variable coefficients in domains
  with moving boundary}, Systems Control Lett. \textbf{80} (2015), 30--41.

\bibitem{Miranda}
Manuel Milla~Miranda, \emph{Exact controllability for the wave equation in
  domains with variable boundary}, Rev. Mat. Univ. Complut. Madrid \textbf{9}
  (1996), no.~2, 435--457.

\bibitem{Moore}
Gerald~T. Moore, \emph{Quantum {T}heory of the {E}lectromagnetic {F}ield in a
  variable-length one-dimensional {C}avity}, Journal of Mathematical Physics
  \textbf{11} (1970), no.~9, 2679--2691.

\bibitem{Myint:PDE}
T.~Myint-U and L.~Debnath, \emph{Linear partial differential equations for
  scientists and engineers}, Birkh{\"a}user Boston, 2007.

\bibitem{Nicolai}
E.L. Nicolai, \emph{On transverse vibrations of a portion of a string of
  uniformly variable length}, Annals Petrograd Polytechn. Inst. \textbf{28}
  (1921), 329--343.

\bibitem{Russell}
David~L. Russell, \emph{Exact boundary value controllability theorems for wave
  and heat processes in star-complemented regions}, Differential games and
  control theory ({P}roc. {NSF}---{CBMS} {R}egional {R}es. {C}onf., {U}niv.
  {R}hode {I}sland, {K}ingston, {R}.{I}., 1973), Dekker, New York, 1974,
  pp.~291--319. Lecture Notes in Pure Appl. Math., Vol. 10.

\bibitem{SunLiLu}
Haicong Sun, Huifen Li, and Liqing Lu, \emph{Exact controllability for a string
  equation in domains with moving boundary in one dimension}, Electron. J.
  Differential Equations (2015), No. 98, 7.

\bibitem{Szekeres.acta.1958}
G.~Szekeres, \emph{Regular iteration of real and complex functions}, Acta
  Mathematica \textbf{100} (1958), no.~3, 203--258.

\bibitem{Yao}
Peng-Fei Yao, \emph{On the observability inequalities for exact controllability
  of wave equations with variable coefficients}, SIAM J. Control Optim.
  \textbf{37} (1999), no.~5, 1568--1599 (electronic).

\bibitem{Zuazua:1Dwave}
E.~Zuazua, \emph{Exact controllability for semilinear wave equations in one
  space dimension}, Ann. Inst. H. Poincar\'e Anal. Non Lin\'eaire \textbf{10}
  (1993), no.~1, 109--129.

\end{thebibliography}
\end{document}